\documentclass[12pt,reqno, final]{amsart}

\usepackage{amsthm,amsmath,amssymb,booktabs,xcolor,graphicx}
\usepackage{bbm}
\usepackage{listings}
\usepackage{xcolor}
\usepackage{appendix}
\usepackage{enumerate}
\usepackage[shortlabels]{enumitem}
\usepackage{verbatim}

\usepackage[margin=1in]{geometry}
\usepackage[utf8]{inputenc}
\usepackage[T1]{fontenc}

\usepackage[textsize=small,backgroundcolor=orange!20]{todonotes}
\usepackage[colorlinks=true, allcolors=blue]{hyperref}
\usepackage{url}
\usepackage{etoolbox}
\usepackage{appendix}
\usepackage[nameinlink, noabbrev,capitalize]{cleveref}
\crefname{equation}{}{} 
\AtBeginEnvironment{appendices}{\crefalias{section}{appendix}} 

\usepackage[color]{showkeys} 

\colorlet{refkey}{orange!20}
\colorlet{labelkey}{blue!30}

\numberwithin{equation}{section}
\newtheorem{theorem}{Theorem}[section]
\newtheorem{proposition}[theorem]{Proposition}
\newtheorem{lemma}[theorem]{Lemma}

\crefname{claim}{Claim}{Claims}

\newtheorem*{question*}{Question}

\theoremstyle{definition}

\newtheorem*{definition*}{Definition}

\theoremstyle{remark}
\newtheorem*{remark}{Remark}

\newcommand{\norm}[1]{\left\lVert#1\right\rVert}
\newcommand{\snorm}[1]{\lVert#1\rVert}

\newcommand{\sang}[1]{\langle #1 \rangle}

\newcommand{\mb}{\mathbb}

\newcommand{\mbm}{\mathbbm}
\newcommand{\mc}{\mathcal}
\newcommand{\mr}{\mathrm}

\newcommand{\on}{\operatorname}

\allowdisplaybreaks

\title{On the real Davies' conjecture}
\author[Jain]{Vishesh Jain}
\author[Sah]{Ashwin Sah}
\author[Sawhney]{Mehtaab Sawhney}
\email{vishesh.vj@gmail.com, \{asah,msawhney\}@mit.edu}
\address{Department of Mathematics, Massachusetts Institute of Technology, Cambridge, MA 02139, USA}

\begin{document}
\begin{abstract}
We show that every matrix $A \in \mb{R}^{n\times n}$ is at least $\delta$$\|A\|$-close to a \emph{real} matrix $A+E \in \mb{R}^{n\times n}$ whose eigenvectors have condition number at most $\tilde{O}_{n}(\delta^{-1})$. In fact, we prove that, with high probability, taking $E$ to be a sufficiently small multiple of an i.i.d. \emph{real} sub-Gaussian matrix of bounded density suffices. This essentially confirms a speculation of Davies, and of Banks, Kulkarni, Mukherjee, and Srivastava, who recently proved such a result for i.i.d. \emph{complex Gaussian} matrices. 

Along the way, we also prove non-asymptotic estimates on the minimum possible distance between any two eigenvalues of a random matrix whose entries have arbitrary means; this part of our paper may be of independent interest. 
\end{abstract}

\maketitle

\section{Introduction}\label{sec:introduction}
Recall that a matrix $A \in \mb{C}^{n\times n}$ is said to be \emph{diagonalizable} if there exists a diagonal matrix $D \in \mb{C}^{n\times n}$ and an invertible matrix $W \in \mb{C}^{n\times n}$ such that $A = WDW^{-1}$. Recall also that $A$ is said to be \emph{normal} if $AA^{\dagger} = A^{\dagger}A$, where $A^{\dagger}$ denotes the Hermitian adjoint of $A$. It is not hard to see that $A$ is normal if and only if it is unitarily diagonalizable i.e. if and only if $A = WDW^{-1}$ for some $W$ satisfying $W^{\dagger} = W^{-1}$. Since $W^{\dagger} = W^{-1}$ implies that $\|W\|\|W^{-1}\| = 1$ (here, $\|\cdot \|$ denotes the standard $\ell^{2} \to \ell^{2}$ operator norm of a matrix), a natural way of quantifying `how far' a matrix is from being normal (equivalently, `how far' it is from being unitarily diagonalizable) is via the \emph{eigenvector condition number} of $A$:
$$\kappa_V(A):= \inf_{W: A = WDW^{-1}}\|W\|\|W^{-1}\|.$$

\subsection{Davies' conjecture: }Any analytic function $f(A)$ of a diagonalizable matrix $A = WDW^{-1}$ may be written as $f(A) = W f(D) W^{-1}$. Note that since $D$ is diagonal, $f(D)$ is the diagonal matrix with non-trivial entries given by $f(D)_{ii} = f(D_{ii})$. However, computationally, this may not be an appropriate way of evaluating $f(A)$ when $A$ is `highly non-normal' i.e. it has a large eigenvector condition number; the reason for this is that even if all computations are carried to precision $\epsilon$, the overall result may be off by an error of order $\kappa_{V}(A)\epsilon$. In \cite{Dav07}, Davies suggested a way around this obstacle, namely, find a small perturbation $A+E$ of $A$ such that $\kappa_{V}(A+E)$ is small, and compute $f(A+E)$ as a substitute for $f(A)$. In order to quantify the error associated to such a scheme, he introduced a quantity called the \emph{accuracy of approximate diagonalization} defined (for matrices $A\in \mb{C}^{n\times n}$ such that $\|A\|\leq 1$) by
$$\underline{\sigma}(A,\varepsilon) := \inf_{E \in \mb{C}^{n\times n}}(\kappa_{V}(A+E)\varepsilon + \|E\|).$$
Davies conjectured that for every positive integer $n$, there exists $c_n$ such that for all $A\in \mb{C}^{n\times n}$ with $\|A\|\leq 1$ and $\varepsilon > 0$, 
\begin{equation}
\label{eqn:davies-conjecture}
\underline{\sigma}(A,\varepsilon) \leq c_n \sqrt{\varepsilon}.
\end{equation}
While Davies was able to confirm his conjecture for a number of special choices of the matrix $A$, for general $A \in \mb{C}^{n\times n}$, he was only able to obtain the following much weaker bound:
\begin{equation}
    \label{eqn:davies-bound}
    \underline{\sigma}(A,\varepsilon) \leq (1+n)\varepsilon^{2/(n+1)}. 
\end{equation}

Recently, Davies' conjecture was resolved by Banks, Kulkarni, Mukherjee, and Srivastava \cite{BKMS19} in a stronger form -- they showed that for every $A \in \mb{C}^{n\times n}$ and every $\delta \in (0,1)$, there exists $E \in \mb{C}^{n\times n}$ such that $\|E\|\leq \delta \|A\|$ and
\begin{equation}
    \label{eqn:bkms-bound}
    \kappa_{V}(A+E)\leq 4n^{3/2}(1+\delta^{-1}). 
\end{equation}
\cref{eqn:davies-conjecture} follows easily from this, for instance, by taking $\delta = 2n^{3/4}\sqrt{\varepsilon}$. In fact, \cite{BKMS19} showed that choosing $E$ from the distribution $G_{n}(\mc{N}_{\mb{C}}(0,1))$, by which we mean an $n\times n$ random matrix, each of whose entries is an independent copy of a complex Gaussian variable with mean $0$ and variance $1/n$, succeeds with high probability.\\ 

Banks et al. asked whether similar results continue to hold if one replaces the complex Gaussian perturbations with a different class of random perturbations, noting that their proof relied on special properties of the complex Gaussian distribution. As a warm-up to our main result, we show that one may indeed replace the standard complex distribution with mean $0$ and variance $1/n$ by $n^{-1/2}\xi$ for any sub-Gaussian complex random variable $\xi$ with bounded density. This is the content of \cref{thm:complex} in \cref{sec:general-complex}.

\subsection{The real Davies' conjecture: }Much more challenging is the question of whether one can replace the complex Gaussian distribution by a real distribution, for instance, the real Gaussian distribution with entries of mean $0$ and variance $1/n$. This is perhaps most interesting when the matrix $A$ is itself a real matrix. While experimental evidence by Davies suggested that this should be possible, Banks et al. adopted a more cautious position, writing that \emph{a proof (or disproof) remains to be found}. 

The key technical challenge in working with real perturbations is that the eigenvalues/eigenvectors of a real matrix are still, in general, complex. A crucial step in \cite{BKMS19} uses the fact that the probability of a standard complex Gaussian lying in any ball of radius $\varepsilon > 0$ is $O(\varepsilon^{2})$. The exponent $2$ in this bound, being the same as the real dimension of the complex plane, allows their argument to go through. However, as noted in \cite[Remark 3.4]{BKMS19}, the analogous argument for real Gaussians completely fails, since the probability of a standard real Gaussian lying in a ball of radius $\varepsilon$ centered around the origin in the complex plane is $\Theta(\varepsilon)$. We discuss this in more detail in \cref{sub:difficulties}.\\ 

Nevertheless, as our main result, we show that real perturbations suffice to regularize the eigenvector condition number (at least when the initial matrix $A$ is real). 

\begin{theorem}\label{thm:main}
There is an absolute constant $C > 0$ such that the following holds. Suppose $\delta\in(0,1/2)$ and $A\in\mb{R}^{n\times n}$. Then there is a matrix $E\in\mb{R}^{n\times n}$ with $\snorm{E}\le\delta\snorm{A}$ such that
\[\kappa_V(A+E)\le Cn^{2}\delta^{-1}\sqrt{\log(n\delta^{-1})}.\]
\end{theorem}
\begin{remark}
As before, this implies Davies' conjecture \cref{eqn:davies-conjecture}, \emph{up to an overall factor of} $(\log(1/\varepsilon))^{1/4}.$
\end{remark}
\begin{remark}
As in \cite{BKMS19}, our proof actually shows that choosing $E$ to be a suitable rescaling of a real random matrix, each of whose entries is an independent copy of a sub-Gaussian real random variale $\xi$ of bounded density, succeeds with high probability. For the precise statement, see \cref{thm:real}.   
\end{remark}

We note that prior to our work, \emph{no result} on the regularization of the eigenvector condition number of general matrices by a \emph{real} translation seems to be known. Indeed, even the much weaker bound \cref{eqn:davies-bound} relies on a theorem of Friedland \cite{Fri01}, which crucially requires working with general complex matrices. 
\subsection{Non-asymptotic bounds on eigenvalue gaps: }One of the steps in our proof of \cref{thm:main} utilizes and extends quite recent tools in the non-asymptotic theory of random matrices, most notably those involved in Ge's \cite{Ge17} resolution of the following seemingly innocuous (but long-standing) problem: show that, with high probability, all the eigenvalues of a random matrix, each of whose entries is independently $\pm 1$ with equal probability, are distinct. While the (easier) Hermitian analog of this had been known (see \cite{tao2017random} and follow-up work), a key challenge in resolving the non-Hermitian case is the same phenomenon we encounter -- non-Hermitian matrices, in general, have complex eigenvalues, whereas the probability for a real random variable to lie in a disc of radius $\varepsilon$ in the complex plane can very well be $\Omega(\varepsilon)$ (as opposed to the `desired' bound of $O(\varepsilon^{2})$).

Building on ideas introduced by Rudelson and Vershynin in \cite{rudelson2016no}, Ge found an ingenious way to circumvent this issue, and the ideas in \cite{Ge17} will also play an important role for us (however, note that by themselves, these ideas unfortunately still do not suffice to derive the main eigenvector condition number regularization result, see the discussion in \cref{sub:difficulties}). Indeed, a part of our argument involves the derivation of non-asymptotic bounds on the size of the eigenvalue gap (the minimum distance between any two eigenvalues) of a random matrix, much like the main result in \cite{Ge17}. However, unlike in \cite{Ge17}, where the entries of the random matrix under consideration are centered, the entries of our random matrices can be \emph{arbitrarily uncentered}. This presents an obstacle in the strategy of \cite{Ge17}, which relies on control of the largest singular value (i.e. operator norm) of the matrix (even when the entries are continuous with bounded density, as will be the case here). 

However, we show that by exploiting control on the \emph{smallest} singular value of the matrix (which, crucially, is unaffected by the mean profile), we can derive a bound on the eigenvalue gap of arbitrarily uncentered random matrices. Whereas the eigenvalues of uncentered, non-normal random matrices have attracted much attention (see, for instance, \cite{guionnet2014convergence, noy2015regularization}, and the references therein), to our knowledge, our result is the first to obtain \emph{any} non-asymptotic bound on the eigenvalue gap for unrestricted mean profiles. This result may be of independent interest.  
\begin{theorem}[Informal, see \cref{prop:spacing} for a precise version]
\label{thm:eigenvalue-gaps}
Let $\xi$ be a sub-Gaussian real random variable with bounded density, and let $G$ denote an $n\times n$ random matrix, each of whose entries is an independent copy of $\xi$. Then, for any $A \in \mb{R}^{n\times n}$ \emph{(}such that $\|A\| \geq 1$\emph{)}, and $s\leq 1$,
$$\mb{P}[\min_{i\neq j}|\lambda_i (A+G) - \lambda _j (A+G)| \leq s] = \tilde{O}\left(s\cdot \|A\|^{4}n^{11/2} + c^{n}\right), $$
where $\lambda_{1}(A+G),\dots, \lambda_{n}(A+G)$ denote the eigenvalues of $A+G$, and $c\in (0,1)$ is a constant depending only on $\xi$.  
\end{theorem}

We remark that in the case of sub-Gaussian complex random variables with bounded density, we obtain a better bound. In particular, for all mean profiles $A \in \mb{C}^{n\times n}$ with norm larger than a small polynomial, our bound improves on a recent result of Banks et al. \cite{BVKS19} (which was proved only for complex Gaussians). See \cref{thm:spacing-complex-general} in \cref{app:complex-spacing} for details. 

\subsection{Organization: }The rest of this paper is organized as follows. In \cref{sec:preliminaries}, we introduce some prelimary notions; in \cref{sec:general-complex}, we present a generalization of the main result in \cite{BKMS19} to general complex random variables; in \cref{sub:difficulties}, we present a more detailed discussion of the difficulties encountered in handling real perturbations, as well as a brief overview of our proof; \cref{sec:strategy} presents the formal statements of the three key estimates required for our proof, and \cref{sec:deduction} shows how to deduce \cref{thm:main} from these estimates, using a bootstrapping argument. Finally, \cref{sec:complex-bound} and \cref{app:spacing} contain the proofs of these three key estimates. We also include three appendices: \cref{app:two-singular-values} contains a standard reduction of the problem of bounding the smallest two singular values of a random matrix, \cref{app:complex-spacing} discusses our improved bounds on the eigenvalue gaps of complex random matrices, and \cref{app:details} reproduces some details from \cite{BKMS19}, needed to prove \cref{thm:main}, for the reader's convenience. 

\subsection{Acknowledgements: }V.J. would like to thank Archit Kulkarni and Nikhil Srivastava for introducing him to the problem.   

\subsection{Concurrent and independent work: }Right before uploading our manuscript to the arXiv, we learned of concurrent and independent work of Banks, Garza-Vargas, Kulkarni, and Srivastava \cite{BVKS20} with similar main results as ours. Both of our works make use of techniques in \cite{BKMS19,Ge17}. However, beyond this similarity, the works are substantially different -- in particular, \cite{BVKS20} works with the limiting expression in \cref{lem:vol-limit} (this requires significant additional ideas, such as developing a version of the restricted invertibility property), whereas our work completely avoids this via a novel bootstrapping scheme. \cite{BVKS20} also works with the first moment of the \emph{eigenvalue overlaps} on the real line (i.e., $\sum_{\lambda_i \in \mb{R}}\kappa(\lambda_i)$), whereas we work exclusively with the (conditional) second moment of \emph{all} the eigenvalue overlaps (i.e. the conditional expectation of $\sum_{\lambda_i \in \mb{C}}\kappa^{2}(\lambda_i)$) (see \cite{BVKS20} for further discussion). 

\begin{itemize}
\item Compared to \cite{BVKS20}, our \cref{thm:main} obtains the near optimal dependence of $\delta^{-1}\log({\delta^{-1}})$ (indeed, \cite{BKMS19} proved a lower bound of $\delta^{-1+1/n}$, even using deterministic complex perturbations, instead of just random real perturbations), whereas \cite{BVKS20} obtains a dependence of $\delta^{-3/2}$. For the real Davies' conjecture \cref{eqn:davies-conjecture}, our bound leads to a resolution (up to log-factors), whereas \cite{BVKS20} obtain the weaker dependence of $\varepsilon^{2/5}$.

\item For quantitative estimates on the eigenvalue gaps, our \cref{thm:eigenvalue-gaps} and \cref{prop:spacing} have a better rate of $\tilde{O}(s)$ (i.e. near linear in the gap size), whereas \cite{BVKS20} achieve a weaker rate of $\tilde{O}(s^{1/3})$. 

\item Compared to our work, \cite{BVKS20} derive logarithmically better quantitative estimates on the first moment of the eigenvalue overlaps on the real line. They also derive non-asymptotic estimates with improved rates for the lower tails of intermediate singular values of random matrices with arbitrary mean profiles; this subject is not considered in our work. 

\item Finally, we remark that since we do not work with the limiting expression in \cref{lem:vol-limit}, our techniques are likely to shed light on \cite[Problem~7.1]{BVKS20}; we intend to return to this in future work.
\end{itemize}

\section{Preliminaries}\label{sec:preliminaries}
\subsection{Random variables}\label{sub:random-variables}
A (complex) random variable $X$ is said to be \emph{sub-Gaussian} if there exist constants $c,C>0$ such that for all $t\geq 0$,
\[\mb{P}[|X| > t]\le Ce^{-ct^2}.\]
The collection of sub-Gaussian random variables on an underlying probability space form a normed space under the \emph{sub-Gaussian norm}
\[\snorm{X}_{\psi_2} := \inf\{t > 0: \mb{E}[e^{(X/t)^2}-1]\le 1\}.\]
For more on sub-Gaussian random variables, we refer the reader to \cite{Ver18}. 
Throughout this paper, we will assume that $\xi$ is a sub-Gaussian random variable. In \cref{sec:general-complex}, it will be complex-valued, and admit a density function that is continuous and bounded by some constant $K$. Otherwise, it will be real-valued, having a density function that is continuous and bounded by some constant $K$.

Finally, we define $G_n(\xi)$ to be the random matrix such that $n^{1/2}G_n(\xi)$ is distributed as an $n\times n$ matrix, each of whose entries is an independent copy of $\xi$.

\subsection{Hyperplane distances}\label{sub:distances}
We will use $\on{dist}(X,H)$ to denote the distance of a vector $X$ to a hyperplane $H$. Note that we will often be considering the distance between complex vectors $X$ and complex hyperplanes $H$, so more formally, the distance is defined as
\[\on{dist}(X,H) := \snorm{\on{proj}_{H^\perp}X}_{2},\]
where $H^\perp$ is the orthogonal complement of $H$ in $\mb{R}^{n}$ or $\mb{C}^{n}$ and $\on{proj}$ is the usual Euclidean projection. 
In particular, if $Z$ has unit length and is normal to $H$, we have
\[\on{dist}(X,H)\ge |\sang{X,Z}|,\]
with equality when $H$ has codimension $1$ (and more generally, when $\on{proj}_{H^\perp}X$ and $Z$ are scalar multiples).

\subsection{Condition numbers and the pseudospectrum}\label{sub:condition-numbers}
Given any $n\times n$ invertible matrix $M$, we define its \emph{condition number} to be $\kappa(M) := \sigma_1(M)/\sigma_n(M) = \snorm{M}\snorm{M^{-1}}$, where $\sigma_1(M)\ge\cdots\ge\sigma_n(M)$ are the singular values of $M$, i.e. the eigenvalues of $(M^\dagger M)^{1/2}$ ($\dagger$ denotes the conjugate transpose), and $\snorm{M} = \sigma_1(M)$ is the standard $\ell^{2} \to \ell^{2}$ operator norm of $M$.

Given a diagonalizable matrix $M$, we define its \emph{eigenvector condition number} to be
\[\kappa_V(M) := \inf_{W: M = WDW^{-1}}\snorm{W}\snorm{W^{-1}}.\]
Furthermore, suppose $M$ has distinct eigenvalues, and  consider its spectral expansion
\[M := \sum_{i=1}^n\lambda_iv_iw_i^\dagger = VDV^{-1},\]
where the right and left eigenvectors $v_i$ and $w_i^\dagger$ are the columns and rows of $V$ and $V^{-1}$, respectively, normalized so that  $w_i^\dagger v_i = 1$. Then, we define the \emph{eigenvalue condition number} of $\lambda_i$ by
\[\kappa(\lambda_i,M) := \snorm{v_iw_i^\dagger} = \snorm{v_i}_2\snorm{w_i}_2.\]
We will also make use of the quantity
\[\kappa_2(M) := \sqrt{\sum_{i=1}^n\kappa(\lambda_i,M)^2}.\]
We will often suppress the dependence on $M$, writing in particular $\kappa_2$ and $\kappa(\lambda_i)$ when the underlying matrix is clear from context.

Next, for any $\varepsilon \geq 0$, we define the $\varepsilon$\emph{-pseudospectrum} of an $n\times n$ matrix $M$ to be
\[\Lambda_\varepsilon(M) := \{z\in\mb{C}: \sigma_n(M-zI)\le\varepsilon\} = \{z\in\mb{C}: \snorm{(M-zI)^{-1}}\ge\varepsilon^{-1}\}.\]
Note, in particular, that $\Lambda_0(M)$ is the spectrum of $M$. 
For a comprehensive treatment of pseudospectra, we refer the reader to the book of Trefethen and Embree.

We will also need the following equivalent characterization of the $\varepsilon$-pseudospectrum (see \cite{TE05} for a proof)
\[\Lambda_\varepsilon(M) = \{z\in\mb{C}: M+E\text{ has eigenvalue }z\text{ for some }\snorm{E}\le\varepsilon\}.\]
Finally, denote the minimum eigenvalue gap of $M$ by $\eta(M)$, i.e.
\[\eta(M) := \min_{i\neq j}|\lambda_i-\lambda_j|,\] where
$\lambda_i$ are the eigenvalues (counting multiplicity) of $M$.\\

We will need some estimates related to the eigenvector condition number and  pseudospectral volume from \cite{BKMS19}.
\begin{lemma}[{\cite[Lemma~3.1]{BKMS19}}]\label{lem:kappaV-upper}
Let $M$ be an $n \times n$ matrix with distinct eigenvalues, and let $V$ be the matrix whose columns are the eigenvectors of $M$ normalized to have unit norm. Then, 
\[\kappa(V)\le\sqrt{n\sum_{i=1}^n\kappa(\lambda_i)^2}.\]
\end{lemma}
\begin{remark}
This implies that $\kappa_V(M)\le \sqrt{n}\kappa_2(M)$.
\end{remark}
\begin{lemma}[{\cite[Lemma~3.2]{BKMS19}}]\label{lem:vol-limit}
Let $M$ be an $n\times n$ matrix with distinct eigenvalues $\lambda_1,\ldots,\lambda_n$, and let $B\in \mathbb{C}$ be a measurable open set. Then 
\[\lim_{\varepsilon\to 0}\frac{\on{vol}(\Lambda_\varepsilon(M)\cap B)}{\varepsilon^2} = \pi\sum_{\lambda_i\in B}\kappa(\lambda_i)^2.\]
\end{lemma}

For our treatment of real perturbations, we will also require the following quantitative version of the previous lemma. 

\begin{lemma}\label{lem:vol-bound}
Let $M$ be an $n \times n$ matrix with distinct eigenvalues $\lambda_i$. Furthermore, let \[0 < \varepsilon \le \frac{\eta(M)}{2n\kappa_2(M)}.\]
Then,
\[\frac{\on{vol}(\Lambda_{\varepsilon}(M)\cap \mc{D}(0,2\snorm{M}))}{\varepsilon^2}\ge\frac{\pi}{8}\kappa_2(M)^2.\]
\end{lemma}
\begin{remark}
Recall that $\kappa_2(M)^2 = \sum_{i=1}^n\kappa(\lambda_i, M)^2$.
\end{remark}
\begin{proof}
First, note that for all $\lambda_i$ we have $\mc{D}(\lambda_i,\kappa(\lambda_i)\varepsilon)\subseteq\mc{D}(0,2\snorm{M})$; this follows from the triangle inequality since $|\lambda_i|\le\snorm{M}$, and $\varepsilon\le \eta(M)/\kappa_2 \le \snorm{M}/\kappa(\lambda_i)$. 

Second, note that $\mc{D}(\lambda_i,\kappa(\lambda_i)\varepsilon)$ and $\mc{D}(\lambda_j,\kappa(\lambda_j)\varepsilon)$ are disjoint for $i\neq j$; indeed, $|\lambda_i-\lambda_j|\ge \eta(M)\ge 2\varepsilon \kappa_2 \geq (\kappa(\lambda_i)+\kappa(\lambda_j))\varepsilon$.

Third, by \cite[Equation~(52.11)]{TE05}, we have that for $z\in \mc{D}(\lambda_j,\kappa(\lambda_j)\varepsilon)$,
\[\bigg|\snorm{(zI-M)^{-1}}-\frac{\kappa(\lambda_j)}{|z-\lambda_j|}\bigg|\le \sum_{i\neq j} \frac{\kappa(\lambda_i)}{|z-\lambda_i|}.\]

Let $J := \{j\in[n]: \kappa(\lambda_j)^2\ge\kappa_2^2/(2n)\}$, and observe that \[\sum_{j\in J}\kappa(\lambda_j)^2\ge \frac{1}{2}\kappa_2^2.\]

We claim that for each $j\in J$, 
\[\mc{D}(\lambda_j,\kappa(\lambda_j)\varepsilon/2)\subseteq\Lambda_\varepsilon(M)\cap\mc{D}(0,2\snorm{M}),\]
and moreover, that these discs are disjoint. Given this claim, we are done, since then, \[\frac{\on{vol}(\Lambda_{\varepsilon}(M)\cap \mc{D}(0,2\snorm{M}))}{\varepsilon^2}\ge\frac{\pi}{4}\sum_{j\in J}\kappa(\lambda_j)^2\ge\frac{\pi}{8}\kappa_2^2.\]

We now prove the claim. The disjointness of the discs follows from the second point above, and from the first point above, we know that $\mc{D}(\lambda_j, \kappa(\lambda_j)\varepsilon/2) \subseteq \mc{D}(0,2 \snorm{M})$. Hence, it only remains to prove that $\mc{D}(\lambda_j, \kappa(\lambda_j)\varepsilon/2) \subseteq \Lambda_\varepsilon(M)$. 

Fix some $j \in J$. If $z\in \mc{D}(\lambda_j,\kappa(\lambda_j)\varepsilon/2)$, then in particular, we have $z\in\mc{D}(\lambda_j,|\lambda_i-\lambda_j|/(4n))$ for each $i\neq j$, so that
\[|z-\lambda_j|\le\frac{|z-\lambda_i|}{4n-1}.\]
Using this, the assumption that $j\in J$, and Cauchy-Schwarz, we have that
\begin{align*}
\sum_{i\neq j} \frac{\kappa(\lambda_i)}{|z-\lambda_i|}
&\le \left(\sum_{i=1}^{n}\kappa(\lambda_i)^2\right)^{1/2}\left(\sum_{i\neq j}\frac{1}{|z-\lambda_i|^2}\right)^{1/2}\\
&\le (2n \kappa(\lambda_j)^2)^{1/2}\left(\frac{(n-1)}{(4n-1)^2|z-\lambda_j|^2}\right)^{1/2}\\
&\le \frac{\kappa(\lambda_j)}{2|z-\lambda_j|}.
\end{align*}
Thus, by the third point above, we obtain
\[\snorm{(zI-M)^{-1}}\ge \frac{1}{2}\frac{\kappa(\lambda_j)}{|z-\lambda_j|}\ge \frac{1}{\varepsilon},\]
so that, by definition, $z \in \Lambda_{\varepsilon}(M)$, as desired. 
\end{proof}

\section{Warm-up: General Complex Perturbations}\label{sec:general-complex}
The proof of Davies' conjecture in \cite{BKMS19} shows that, with high probability, the addition of a suitably scaled copy of $G_{n}(\mc{N}_{\mb{C}}(0,1))$, where $\mc{N}_{\mb{C}}(0,1)$ is the standard complex Gaussian, regularizes the eigenvector condition number. Here, as a warm-up to the much more involved real case, we provide a brief outline of how to achieve such a regularization using $G_{n}(\xi)$, where $\xi$ is any sub-Gaussian complex random variable with bounded density. 
More precisely, we will show the following.
\begin{proposition}\label{thm:complex}
Let $\xi$ be a sub-Gaussian complex random variable with bounded density. Given $A\in\mb{C}^{n\times n}$ with $\snorm{A} = 1$ and $\delta\in(0,1)$, there are $C_1,C_2$ depending only on $\xi$ such that we have
\[\mb{P}[\kappa_V(A+\delta G_n)\le C_1 n^2\delta^{-1}\text{ and }\snorm{\delta G_n}\le C_2 \delta]\ge 1/2\]
for $G_n = G_n(\xi)$.
\end{proposition}
\begin{remark}
The constant $C_2$ depends only on the sub-Gaussian moment of $\xi$ and $C_1$ depends on the density bound and sub-Gaussian moment of $\xi$. The sub-Gaussian condition can easily be relaxed to a fourth moment assumption with no cost and to finiteness of the second moment by incurring additional polynomial factors.
\end{remark}

Our outline is identical to the one in \cite{BKMS19}; the only difference is that \cite[Lemma~2.2]{BKMS19} is replaced by \cref{lem:operator-tail}, and \cite[Lemma~3.3]{BKMS19} is replaced by \cref{lem:general-small-ball}.\\    

For the remainder of this section, we will write $G_{n}(\xi)$ simply as $G_{n}$. We begin with a standard bound on the operator norm of an i.i.d. matrix with sub-Gaussian entries.
\begin{lemma}[{From \cite[Theorem~4.6.1]{Ver18}}]\label{lem:operator-tail}
For an absolute constant $K$ depending only on the sub-Gaussian moment of $\xi$,
\[\mb{P}[\sigma_1(G_n)\ge K + t]\le\exp(-nt^2).\]
\end{lemma}
Next, we need a specialization of a lemma regarding the sum of multidimensional bounded density random variables from \cite{MMX17}. The result we require can also be derived from \cite{JL15}.
\begin{lemma}[{From \cite[Theorem~1.1]{MMX17}}]
\label{lem:density-bound}
Given real numbers $a_1,\ldots,a_n$ with $\sum_{i=1}^na_i^2=1$ and $\mb{R}^2$-valued continuous random variables $X_1,\ldots,X_n$ with densities bounded by $K$, we have that
\[a_1X_1+\cdots+a_nX_n\]
is a random variable with density bounded by $eK$.
\end{lemma}

As mentioned above, we need an analogue of  \cite[Lemma~3.3]{BKMS19} for general complex distributions with bounded density. This is a standard exercise in using the `invertibility-via-distance' approach (see, e.g. \cite{Tik17}). 
\begin{lemma}[Small ball estimate for $\sigma_n$]\label{lem:general-small-ball}
For any fixed $A\in\mb{C}^{n\times n}$ and for all $\epsilon  \geq 0$,
\[\mb{P}[\sigma_n(A+\delta G_n)\le\epsilon]\le\pi eKn^3\frac{\epsilon^2}{\delta^2}.\]
\end{lemma}
\begin{proof}
Let $M = A + \delta G_n$, with columns $M_i$ for $1\le i\le n$. Recall that $\sigma_{n}(M)$ admits the following variational characterization:
$$\sigma_{n}(M) = \inf_{x\in \mb{S}^{n-1}}\|Mx\|_{2}.$$

Define hyperplanes $H_i = \on{span}\{M_j: j\neq i\}$. The key point is that any vector $x\in \mb{S}^{n-1}$ has a coordinate, say $x_1$, satisfying $x_1\ge n^{-1/2}$, in which case
\[\snorm{Mx}_2\ge\on{dist}(Mx,H_1) = \on{dist}(x_1M_1,H_1)\ge n^{-1/2}\on{dist}(M_1,H_1).\]
Therefore, by the union bound,
\[\mb{P}[\sigma_n(M)\le\epsilon] = \mb{P}\left[\inf_{x\in \mb{S}^{n-1}}\snorm{Mx}_2\le\epsilon\right]\le\sum_{i=1}^n\mb{P}[\on{dist}(M_i,H_i)\le\epsilon n^{1/2}].\]
Now, for each $1\le i\le n$,
\[\mb{P}[\on{dist}(M_i,H_i)\le\epsilon n^{1/2}] = \mb{E}[\mb{P}[\on{dist}(M_i,H_i)\le\epsilon n^{1/2}|H_i]]\le\pi eKn^2\frac{\epsilon^2}{\delta^2}.\]

We explain the final inequality: $\on{dist}(M_i,H_i) \geq |\langle M_i, N_i\rangle|$, where $N_i$ is any unit vector orthogonal to $H_i$ (in particular, $N_i$ may be chosen measurably with respect to $H_i$). 
For such a measurable choice, $M_i$ and $N_i$ are independent after conditioning on $H_i$, and the random variable $\langle M_i, N_i \rangle$ may be written as $M_{i,1}e^{i\theta_1}|N_{i,1}| + \dots + M_{i,n}e^{i\theta_{n}}|N_{i,n}|$, where $|N_{i,1}|^{2} + \dots + |N_{i,n}|^{2} = 1$, and $M_{i,j}e^{i\theta_j}$ are independent random variables with densities bounded by $Kn\delta^{-2}$. Therefore, by \cref{lem:density-bound}, the density of $\langle M_i, N_i \rangle$ is at most $eKn\delta^{-2}$, so that $\mb{P}[|\langle M_i, N_i \rangle| \leq \epsilon n^{1/2}| H_i] \leq \pi e Kn^{2}\epsilon^{2}\delta^{-2}$. 

Finally, summing over the $n$ events for $1\le i\le n$ gives the desired conclusion.
\end{proof}
We now deduce the analogue of \cite[Thm~1.5]{BKMS19} for general complex random variables $\xi$; the proof is identical modulo substituting \cref{lem:general-small-ball} for \cite[Lemma~3.3]{BKMS19}. We repeat it here in order to highlight several key difficulties in extending the proof to the real case.
\begin{lemma}\label{lem:complex-final}
Suppose $A\in\mb{C}^{n\times n}$ with $\snorm{A}\le 1$ and $\delta\in(0,1)$. Let  $\lambda_1,\ldots,\lambda_n$ be the (random) eigenvalues of $A+\delta G_n$, where $G_{n} = G_{n}(\xi)$. Then, for any open measurable set $\mc{B} \subseteq \mb{C}$, 
\[\mb{E}\sum_{\lambda_i\in\mc{B}}\kappa(\lambda_i)^2\le \frac{eKn^3}{\delta^2}\on{vol}(B).\]
Note that the $\lambda_i$ are distinct with probability $1$, so that the left hand side is indeed well-defined.
\end{lemma}
\begin{proof}
For any measurable open set $\mc{B}\subseteq\mb{C}$ we have
\begin{align}\label{eq:circular-integral}
\begin{split}
\mb{E}\on{vol}(\Lambda_\varepsilon(A+\delta G_n)\cap\mc{B}) 
&= \mb{E}\int_{\mc{B}}\mbm{1}_{\{z\in\Lambda_\varepsilon(A+\delta G_n)\}}\,dz\\
&= \int_{\mc{B}}\mb{P}[z\in\Lambda_\varepsilon(A+\delta G_n)\,dz\\
&= \int_{\mc{B}}\mb{P}[\sigma_n(A+\delta G_n-zI)<\varepsilon]\,dz\\
&\le\pi eKn^3\frac{\varepsilon^2}{\delta^2}\on{vol}(\mc{B})
\end{split}
\end{align}
by \cref{lem:general-small-ball}.

Now \cref{lem:vol-limit} yields
\[\mb{E}\sum_{\lambda_i\in\mc{B}}\kappa(\lambda_i)^2\le\liminf_{\varepsilon\to 0}\frac{\mb{E}\on{vol}(\Lambda_\varepsilon(A+\delta G_n)\cap\mc{B})}{\pi\varepsilon^2}\le \frac{eKn^3}{\delta^2}\on{vol}(B).\qedhere\]
\end{proof}

Given \cref{lem:complex-final,lem:operator-tail} the deduction of \cref{thm:complex} follows exactly as the deduction of \cite[Theorem~1.1]{BKMS19} given \cite[Theorem~1.5, Lemma~2.2]{BKMS19}; in particular, \cref{lem:complex-final} is only needed when $\mc{B}$ is a disk of around constant radius. For the reader's convenience, we reproduce the details in \cref{app:details}. The various remarks regarding moment conditions following \cref{thm:complex} follow from standard versions of \cref{lem:operator-tail} under weaker moment assumptions 

\subsection{Difficulties for real perturbations}\label{sub:difficulties} 
In order to prove that real perturbations suffice, we must overcome significant technical difficulties. As mentioned in the introduction, the major issue is that the small ball estimate \cref{lem:general-small-ball} is in general weakened from an $\varepsilon^2$ rate to an $\varepsilon$ rate 
; morally, in the real case we can only study intervals of length $\varepsilon$, not disks of radius $\varepsilon$. However, since the method of \cite{BKMS19} looks at the limiting volume of the pseudospectrum $\Lambda_\varepsilon(A+\delta G_n)$, the $\varepsilon^2$ rate is required in a fundamental way, as we can see from the proof above. 

In order to fix this issue, we generalize results of Ge \cite{Ge17} to random matrices with an arbitrary mean profile -- as we will see, this provides a $\varepsilon^2$ bound, at the additional cost of a factor of $|\on{Im}z|^{-1}$ in \cref{eq:circular-integral}. Therefore, we will split the integral over the disc into two parts: within a strip of width $\varepsilon$ around the real axis, we use the bound with rate $\varepsilon$, and outside it, we use our generalization of Ge's \cite{Ge17} results.

However, this still leads to an additional factor of $\log(\varepsilon^{-1})$, which blows up if we take the limit $\varepsilon \to 0$ as in the proof of \cref{lem:complex-final} (which comes from \cref{lem:vol-limit}). We may avoid taking the limit by finding an effective value of $\varepsilon$, which still makes \cref{lem:vol-limit} hold; however, it turns out that such an effective value depends both on \emph{a priori} estimates on $\eta(A+\delta G_{n})$ (which can be obtained using \cref{thm:eigenvalue-gaps}) and more seriously, on $\kappa_{V}(A+\delta G_{n})$, which is precisely the quantity we are trying to control!  

As mentioned in the introduction, prior to our work, no such estimate seems to be known, not even when $G_{n}$ is replaced by a deterministic real matrix of bounded operator norm. 
We overcome this difficulty via a novel bootstrapping scheme which completely eliminates the need for any \emph{a priori} bound on $\kappa_{V}(A+\delta G_{n})$. 

\section{Overview of Key Estimates}\label{sec:strategy}
For the remainder of the paper, we assume that $\xi$ is a real sub-Gaussian random variable with density bounded by $K$ and $G_n = G_n(\xi)$. 

Our proof of \cref{thm:main} relies on  three key estimates.\\

First, we require a bound on the lower tail of the smallest singular value of $G_n(\xi)$ which is invariant with respect to translations by arbitrary complex matrices.
\begin{proposition}\label{prop:unif-bound}
For any $A\in\mb{C}^{n\times n}$, we have
\[\mb{P}[\sigma_n(A+\delta G_n)\le\epsilon]\le 2\sqrt{2e}Kn^2\frac{\epsilon}{\delta}.\]
\end{proposition}
In the case when $A \in \mb{R}^{n\times n}$, this essentially appears in \cite{Rud13} and also as \cite[Observation~2]{Tik17}. We provide the proof for general ${A} \in \mb{C}^{n\times n}$ in \cref{sec:complex-bound}.\\

Next, we provide a generalization of the main singular value estimate in Ge \cite{Ge17} to the case when the entries of the random matrix have arbitrary means.
\begin{proposition}\label{prop:complex-sv}
Let $A\in\mb{R}^{n\times n}$ with $\snorm{A}\le 1$ and $\delta \in (0,1)$. Let $\mc{E}_{K'}$ denote the event that $\snorm{G_n}\le K'$, and let $z\in \mb{C}$ such that $|z|\le 3\delta K'+3$. Finally, let $M = A + \delta G_n-zI$. Then
\[\mb{P}[\mc{E}_{K'}\cap\sigma_n(M)\le\epsilon]\le c_{\ref{prop:complex-sv}}(1+\delta K')\frac{\max(K,K^2)n^3\epsilon^2}{\delta^2|\on{Im}z|}\]
where $c_{\ref{prop:complex-sv}} > 0$ is an absolute constant.
\end{proposition}
As in \cite{Ge17}, the crucial part of this bound is the rate $\epsilon^{2}$ (as opposed to just $\epsilon$). The proof of the above proposition requires some essential departures from Ge's proof, which requires strong control over the the ratio $\|M\|/\|G_{n}\|$ (in particular, the proof in \cite{Ge17} breaks down for $\delta = o(\exp(-O(n))$. We provide details in \cref{sec:complex-bound}.\\

From an analogue of the above proposition for the smallest two singular values, we derive the following more precise version of \cref{thm:eigenvalue-gaps} -- this is similar to main result in \cite{Ge17} (but again, generalized to an arbitrary mean profile). We prove it in \cref{app:spacing}.
\begin{proposition}\label{prop:spacing}
Let $A\in\mb{R}^{n\times n}$ with $\snorm{A}\le 1$, let $\delta \in (0,1)$, and let $M = A + \delta G_n$. Let $\mc{E}_{K'}$ denote the event that $\snorm{G_n}\le K'$. Then, for $s\le 1$,
\[\mb{P}[\mc{E}_{K'}\cap\eta(M)\le s]\le c_{\ref{prop:spacing}}\log_2(2(1+\delta K')/s)\cdot (1+\delta K')^2\max(K^2,K^3)n^5\frac{s}{\delta^3}\]
where $c_{\ref{prop:spacing}} > 0$ is an absolute constant.
\end{proposition}

\section{Deduction of \texorpdfstring{\cref{thm:main}}{Theorem 1.1}}\label{sec:deduction}
In this section, we show how to prove \cref{thm:main}, given the estimates in the previous section.\\ 

We begin with an immediate corollary of \cref{prop:spacing}.
\begin{lemma}\label{lem:spacing}
Let $A\in\mb{R}^{n\times n}$ with $\norm{A}\le 1$, and let $\delta\in (0,1)$. Then 
\[\mb{P}[\eta(A+\delta G_n)\ge c_{\xi,\ref{lem:spacing}}n^{-6}\delta^4]\ge\frac{3}{4}\]
for some $c_{\xi,\ref{lem:spacing}} > 0$ depending only on the density bound $K$ and sub-Gaussian moment of $\xi$.
\end{lemma}
\begin{remark}
The $n^{-6}\delta^4$ in the above can be replaced by $n^{-5}\delta^3/\log(n\delta^{-1})$.
\end{remark}
\begin{proof}
Let $\mc{E}_{K'} = \{\snorm{G_n}\le K'\}$ and choose $K'$ large enough (depending on the sub-Gaussian moment of $\xi$) so that
\[\mb{P}[\mc{E}_{K'}]\ge 7/8.\]
This can be done by using \cref{lem:operator-tail}. 

Then, for an appropriate choice of $c_{\xi,\ref{lem:spacing}}$, we obtain
\[\mb{P}[\mc{E}_{K'}\cap\eta(A+\delta G_n)\le c_{\xi}n^{-6}\delta^4]\le 1/8\]
by \cref{prop:spacing}. The result follows by the union bound.
\end{proof}

We are now ready to prove our main technical result, which is an analogue of \cref{thm:complex} for real perturbations.
\begin{theorem}\label{thm:real}
Let $\xi$ be a sub-Gaussian real random variable with bounded density. Given $A\in\mb{R}^{n\times n}$ with $\snorm{A} \le 1$ and $\delta\in(0,1)$, there are $C_1,C_2$ depending only on $\xi$ such that
\[\mb{P}[\{\kappa_V(A+\delta G_n(\xi))\le C_1 n^{2}\delta^{-1}\sqrt{\log(n\delta^{-1})}\}\cap\{\snorm{\delta G_n}\le C_2 \delta\}]\ge 1/2.\]
\end{theorem}
\begin{proof}
For an appropriate constant $K'$ depending only on $\xi$, we have by \cref{lem:operator-tail} that
\[\mb{P}[\snorm{G_n}\ge K']\le 1/12.\]
Hence, by \cref{lem:spacing} and the union bound, we have that the event
\[\mc{E} = \{\eta(A+\delta G_n)\ge c_{\xi,\ref{lem:spacing}} n^{-6}\delta^4\}\cap\{\snorm{G_n}\le K'\}\]
has probability at least $2/3$. 

Define
\[p_t = \mb{P}[\kappa_2(A+\delta G_n)\le t|\mc{E}]\]
(recall that $\kappa_2^2 = \sum_{i=1}^n\kappa(\lambda_i)^2$) and 
\[\mc{E}_t = \mc{E}\cap\{\kappa_2\le t\}.\]

Fix any real $t\geq 1$ with $p_t \geq 9/10$. We will show that for any such $t$, conditioned on the event $\mc{E}_{t}$, the expectation of $\kappa_{2}^{2}$ is significantly smaller than $t$ (in fact, the dependence on $t$ is logarithmic). By Markov's inequality, we can then conclude that $p_{f(t)}$ holds with probability at least $p_{t}(1-O(f(t)^{-1}))$, where $f(t) = \delta^{-2}\log{t}$ (up to logarithmic factors, and factors of $n$; see the precise definition below). In particular, this will allow us to show that starting from any (arbitrarily large) $t_0$ for which $p_{t_0} \geq 99/100$, we can keep iterating this process to find some $t = O(\delta^{-2})$ (suppressing log factors, and factors of $n$) such that $p_t \geq 9/10$, which suffices for our application.\\

To accomplish this, we begin by applying \cref{lem:vol-bound} to $M = A + \delta G_n$ (conditioned on $\mc{E}_{t}$) with $\varepsilon = \min(1,c_{\xi,\ref{lem:spacing}})n^{-6}\delta^4/(2nt)$. Note that conditioned on $\mc{E}_{t}$, we indeed have that
$$0 < \varepsilon \leq \frac{\eta(M)}{2n\kappa_{2}(M)},$$
so that the application of \cref{lem:vol-bound} is valid. Then, 
\begin{align*}
\mb{E}[\kappa_2^2|\mc{E}_t]&\le\frac{8}{\pi\varepsilon^2}\mb{E}[\on{vol}(\Lambda_\varepsilon(A+\delta G_n)\cap\mc{D}(0,2\snorm{M}))|\mc{E}_t]\\
&=\frac{8}{\pi\varepsilon^2}\mb{E}\left[\int_{\mc{D}(0,2\snorm{M})}\mbm{1}_{\{z\in\Lambda_\varepsilon(A+\delta G_n)\}}\,dz\bigg|\mc{E}_t\right]\\
&\le\frac{8}{\pi\varepsilon^2}\mb{E}\left[\int_{\mc{D}(0,2+2\delta K')}\mbm{1}_{\{z\in\Lambda_\varepsilon(A+\delta G_n)\}}\,dz\bigg|\mc{E}_t\right]\\
&=\frac{8}{\pi\varepsilon^2}\int_{\mc{D}(0,2+2\delta K')}\mb{E}[\mbm{1}_{\{z\in\Lambda_\varepsilon(A+\delta G_n)\}}|\mc{E}_t]\,dz\\
&=\frac{8}{\pi\varepsilon^2}\int_{\mc{D}(0,2+2\delta K')}\mb{P}[\sigma_n(A+\delta G_n-zI)\le\varepsilon|\mc{E}_t]\,dz\\
&\le\frac{8}{\pi\varepsilon^2}\int_{\mc{D}(0,2+2\delta K')}\min\left(2\sqrt{2e}Kn^2\frac{5\varepsilon}{3\delta},c_{\ref{prop:complex-sv}}(1+\delta K')\frac{5\max(K,K^2)n^3\varepsilon^2}{3\delta^2|\on{Im}z|}\right)\,dz.
\end{align*}
In the last line we used \cref{prop:unif-bound,prop:complex-sv}, along with $\mb{P}[\mc{E}_t] = \mb{P}[\mc{E}_t|\mc{E}]\mb{P}[\mc{E}]\ge (9/10) \cdot (2/3) = 3/5$.\\ 

Now splitting the disk at $|\on{Im}z| = (1+\delta K')\max(1,K)\varepsilon\delta^{-1}n$ and integrating, we obtain
\[\mb{E}[\kappa_2^2|\mc{E}_t]\lesssim_{\xi} (1+\delta K')^2\max(K,K^2)n^3\delta^{-2} + (1+\delta K')^2\max(K,K^2)n^3\delta^{-2}\log(\delta/(n\varepsilon)).\]
Note that, by our choice of $\varepsilon$, the first term is bounded by the second (up to an absolute constant factor). Thus
\begin{align}
\mb{E}[\kappa_2^2|\mc{E}_t]&\lesssim_{\xi}(1+\delta K')^2\max(K,K^2)n^3\delta^{-2}\log(\delta/(n\varepsilon))\notag\\
&\lesssim_\xi(1+\delta K')^2\max(K,K^2)n^3\delta^{-2}\log(tn/\delta).\label{eq:key-L2-iterator}
\end{align}

By Markov's inequality, setting $f(t) = n^3\delta^{-2}\log(tn/\delta)$, 
\[\mb{P}[\kappa_2\ge f(t)|\mc{E}_t]\lesssim C_\xi(1+\delta K')^2\max(K,K^2)f(t)^{-1}.\]
Thus,
\[p_{f(t)} = \mb{P}[\kappa_2 \leq f(t) | \mc{E}_{t}]\mb{P}[\mc{E}_t]\ge p_t(1-C_\xi'f(t)^{-1}),\]
for $C_\xi' = C_\xi(1+\delta K')^2\max(K,K^2)$ (which is of constant order since $\delta\in(0,1)$).\\

Choose $t_0$ such that $p_{t_0}\ge 99/100$, and let $t_{i+1}=f(t_i)$ as long as
\[t_i\ge 8n^3\delta^{-2}\log(1+n^4\delta^{-3})\text{ and } f(t_i)\ge 2^{16}C_\xi';\]
otherwise terminate. In particular, if \begin{equation}\label{eq:t-bound}
t_i\ge 8n^3\delta^{-2}\log(1+n^4\delta^{-3})+\exp(2^{16}C_\xi'),
\end{equation}
then the process continues for at least one more step.
\\ 

We claim that as long as the process does not terminate,
\[t_i\ge 2n^3\delta^{-2}\log(t_in/\delta) = 2t_{i+1};\]
in particular, the process terminates in finite time. To see the claimed inequality, note that the function
$g(x) = x - 2n^3\delta^{-2}\log(xn/\delta)$ is positive at $x = 8n^3\delta^{-2}\log(1+n^4\delta^{-3})$, and has derivative $g'(x) = 1 - 2n^2\delta^{-1}x^{-1}$, which is also positive for $x\ge 8n^3\delta^{-2}\log(1+n^4\delta^{-3})$.\\

Next, we claim that at every step before termination,  $p_{t_i} \geq 9/10$. 
Indeed, we know from the above discussion that if $p_{t_i}\ge 9/10$, then
\[p_{t_{i+1}} = p_{f(t_i)}\ge p_{t_i}(1-C_\xi't_{i+1}^{-1}).\]
Hence, if the process continues for another step, then
\[p_{t_{i+1}}\ge\frac{99}{100}\prod_{j=1}^{i+1}(1-C_\xi't_j^{-1})\ge\frac{9}{10},\]
so that the lower bound on the probability is maintained throughout the process.
The last inequality follows from
\[t_j\ge 2t_{j+1}\]
for $0\le j\le i$ and $t_{i+1}\ge 2^{16}C_\xi'$, which together imply
\[\prod_{j=1}^{i+1}(1-C_\xi't_j^{-1})\ge\prod_{j=0}^\infty(1-2^{-j-16})\ge\frac{10}{11}.\]

This, together with \cref{eq:t-bound}, implies that there exists some
$t\leq 8n^3\delta^{-2}\log(1+n^4\delta^{-3})+\exp(2^{16}C_\xi')$ such that
\cref{eq:key-L2-iterator} holds for this value $t$.\\

Using \cref{lem:kappaV-upper} for such a choice of $t$, we obtain
\[\mb{E}[\kappa_V(A+\delta G_n)^2|\mc{E}_t]\lesssim_\xi(1+\delta K')^2\max(K,K^2)n^4\delta^{-2}\log(tn/\delta).\]
Another application of Markov's inequality gives
\[\mb{P}[\kappa_V(A+\delta G_n)\ge C_\xi''(1+\delta K')\max(K^{1/2},K)n^{2}\delta^{-1}\sqrt{\log(tn/\delta)}|\mc{E}_t]\le 1/6\]
for sufficiently large $C_\xi''$. Since $\mb{P}[\mc{E}_t]\ge 3/5$, we conclude that
\[\mb{P}[\snorm{G_n}\le K'\cap\kappa_V(A+\delta G_n)\le\sqrt{6}(1+\delta K')\max(K^{1/2},K)n^{2}\delta^{-1}\sqrt{\log(tn/\delta)}]\ge 1/2.\]
The result now follows from substituting the upper bound on $t$ corresponding to \cref{eq:t-bound}.
\end{proof}

\begin{proof}[Proof of \cref{thm:main}]
This follows immediately from \cref{thm:real}. 
\end{proof}

\section{Proof of \texorpdfstring{\cref{prop:unif-bound,prop:complex-sv}}{Propositions 4.1 and 4.2}}\label{sec:complex-bound}

\subsection{Invertibility via distance}\label{sub:invertibility-via-distance}
We begin with a (by now) standard reduction of the singular value estimates to estimates on the distance of a random vector to an appropriate subspace. 

\begin{lemma}\label{lem:reduction}
Let $M$ be any $n\times n$ random matrix and fix any $\epsilon \ge 0$. Let $M_i$ be the $i^{th}$ column vector of $M$ and $H_i = \on{span}\{M_j: j\neq i\}$. Then for any event $\mc{E}$,
\[\mb{P}[\sigma_n(M)\le \epsilon\cap \mc{E}]\le \sum_{i=1}^{n}\mb{P}[\on{dist}(M_i,H_i)\le \epsilon n^{1/2}\cap \mc{E}].\]
\end{lemma}
\begin{proof}
The argument is identical to the first half of the proof of \cref{lem:general-small-ball}, but intersected with an arbitrary event.
\end{proof}
\begin{lemma}\label{lem:double-reduction}
Let $M$ be any $n\times n$ random matrix and fix any $\epsilon_2\ge\epsilon_1\ge 0$. Let $M_i, H_i$ be as in \cref{lem:reduction} and let $H_{i,j} = \on{span}\{M_k: k\neq i,j\}$. Then for any event $\mc{E}$,
\[\mb{P}[\sigma_n(M)\le\epsilon_1\cap\sigma_{n-1}(M)\le\epsilon_2\cap \mc{E}]\le \sum_{i=1}^{n}\sum_{j\neq i}\mb{P}[\on{dist}(M_i,H_i)\le \epsilon_1 n^{1/2}\cap \on{dist}(M_j,H_{i,j})\le \epsilon_2 n^{1/2}\cap \mc{E}].\]
\end{lemma}
The proof of \cref{lem:double-reduction} is deferred to \cref{app:two-singular-values}; this also completely standard.

\subsection{Proof of \texorpdfstring{\cref{prop:unif-bound}}{Proposition 4.1}}
The proof of \cref{prop:unif-bound} is quite simple, given the following lemma 
of Livshyts, Paouris, and Pivovarov \cite{LPP16}. 
\begin{lemma}[{From \cite[Theorem~1.1]{LPP16}}]\label{lem:low-density}
Consider independent $\mb{R}$-valued continuous random variables $X_1,\ldots,X_n$ with densities bounded by $K$ and let $X= (X_1,\ldots,X_n)$. Fix an integer $\ell\in \{1,2\}$. Then, for any $V\in \mb{R}^{\ell\times n}$, $VX$
is an $\mb{R}^\ell$-valued random variable with density bounded by 
\[\frac{e^{\ell/2}K^{\ell}}{\det(VV^{\mr{T}})^{1/2}}.\]
\end{lemma}
This has the following immediate corollary. 
\begin{lemma}\label{lem:universal-hyperplane}
For any $A\in\mb{C}^{n\times n}$, let $M = A + \delta G_n$. For all $j\in [n]$, if $v_j$ is a unit normal to $H_j = \on{span}\{M_k: k\neq j\}$, then
\[\mb{P}[|\sang{M_j,v_j}|\le\epsilon|M_{-j}]\le 2\sqrt{2e}Kn^{1/2}\frac{\epsilon}{\delta}.\]
Here, $M_{-j}$ denotes the set of all columns excluding $M_{j}$. 
\end{lemma}
\begin{proof}
After possibly multiplying by an overall phase, we may assume that $v_j = x_j + iy_j$ satisfies $\snorm{x_j}_2\ge 1/\sqrt{2}$. Note that multiplication by an overall phase does not affect the random variable $|\langle M_j, v_j \rangle|$.   
Then,
\begin{align*}
\mb{P}[\on{dist}(M_j,H_j)\le\epsilon |M_{-j}] &\leq \mb{P}[|\sang{M_j,v_j}|\le\epsilon |M_{-j}] = \mb{P}[|\sang{A_j,v_j}+\delta\sang{(G_n)_j,v_j}|\le\epsilon |M_{-j}]\\
&\le\mb{P}[|\on{Re}(\sang{A_j,v_j})+\delta\sang{(G_n)_j,x_j}|\le\epsilon |M_{-j}]\\
&\le 2\sqrt{2e}Kn^{1/2}\frac{\epsilon}{\delta}.
\end{align*}
For the last inequality, note that $\delta(G_n)_j$ is a vector of independent real random variables, each with density bounded by $Kn^{1/2}\delta^{-1}$, so that applying \cref{lem:low-density} to this vector with $V = x_{j}^{\mr{T}}$, and using $\|x_{j}\|_{2} \geq 1/\sqrt{2}$ to control the term $\det(VV^{\mr{T}})$ gives the desired conclusion.
\end{proof}
Now we are ready to deduce \cref{prop:unif-bound}.
\begin{proof}[Proof of \cref{prop:unif-bound}]
Let $M = A + \delta G_n$, with columns $M_j$ for $1\le j\le n$. Define hyperplanes $H_j = \on{span}\{M_k: k\neq j\}$. Using \cref{lem:reduction}, we have
\[\mb{P}[\sigma_n(M)\le\epsilon]\le\sum_{j=1}^n\mb{P}[\on{dist}(M_j,H_j)\le\epsilon n^{1/2}].\]
For each $1\le j\le n$, choose a unit normal vector $z_j$ of $H_j$ (independently of $M_j$). Then, \cref{lem:universal-hyperplane} gives
\[\mb{P}[\on{dist}(M_j,H_j)\le\epsilon n^{1/2}|M_{-j}] \le 2\sqrt{2e}Kn\frac{\epsilon}{\delta}.\]
Since this probability is uniform over the realization of $M_{-j}$, we may remove the conditioning using the law of total probability to finish the proof. 
\end{proof}

\subsection{Proof of \texorpdfstring{\cref{prop:complex-sv}}{Proposition 4.2}}
The following is the key step in the proof. 
\begin{lemma}\label{lem:two-dimensionality}
Let $A\in\mb{R}^{n\times n}$ with $\snorm{A}\le 1$ and $\delta \in (0,1)$. Let $z\in \mb{C}$ such that $|z|\le 3\delta K'+3$ and $M = A + \delta G_n-zI$. Let $\mc{E}_{K'}$ denote the event that $\snorm{G_n}\le K'$, let $H_j=\on{span}\{M_k:k\neq j\}$, and let $v_j$ be any unit normal vector of $H_j$. Then,
\[\mb{P}[\mc{E}_{K'} \cap|\sang{M_j,v_j}| \le \epsilon|M_{-j}]\le c_{\ref{lem:two-dimensionality}}(1+\delta K')\frac{\max(K,K^2)n\epsilon^2}{\delta^2|\on{Im}z|},\]
where $c_{\ref{lem:two-dimensionality}}$ is an absolute constant. 
\end{lemma}
\begin{proof}
Let $B := A + \delta G_n - \on{Re}(z)I$. 
For convenience of notation, let $\rho = \on{Im}z$. We split into cases based on $|\rho|$ versus $\epsilon$.\\

\textbf{Case I: }$|\rho| \leq 4\epsilon$. From \cref{lem:universal-hyperplane} applied to $A - zI$, we find that
\[\mb{P}[|\sang{M_j,z_j}|\le\epsilon|M_{-j}]\le 2\sqrt{2e}Kn^{1/2}\frac{\epsilon}{\delta}\le 2\sqrt{2e}Kn\frac{4\epsilon^2}{|\rho|\delta^2},\]
since $\delta < 1$. This finishes the proof in this case. \\

\textbf{Case II: }$|\rho| > 4\epsilon$. We initially show, restricted to $\mc{E}_{K'}$, that no $v := x+iy \in \mb{S}^{n-1}$ with $\snorm{y}_2 < |\rho|/(16\delta K'+16)$ can satisfy
\[\snorm{Mv}_2 < \epsilon.\]
Indeed, if this were true, then writing down the equation for the imaginary part of $Mv$, and using $M = B - i\rho I$, where $B,\rho$ are real, we get that
\[\snorm{By - \rho x}_2 < \epsilon.\]
Thus, we would have
\[\frac{|\rho|}{2}<\snorm{\rho x}_2 < \epsilon + \snorm{By}_2 < \epsilon + \frac{|\rho|}{4},\]
where the first inequality uses $\snorm{y}_2\le 1/6$ to deduce $\snorm{x}_2\ge 1/2$, the second inequality is simply the triangle inequality, and the final inequality uses that $\snorm{B}\le 1 + K'\delta + 3(\delta K'+1)\le 4(\delta K'+1)$. But this is a contradiction to $|\rho| > 4\epsilon$.

Now note that if $\mc{E}_{K'}\cap|\sang{M_j,v_j}|\le\epsilon$ occurs, then \[\snorm{Mv_j}_2 = |\sang{M_j,v_j}|\le\epsilon,\]
which implies that $v_j=x_j+iy_j$ satisfies $\snorm{y_j}_2\ge|\rho|/(16\delta K'+16)$. In fact, by shifting $v_j$ by a phase in the above argument, we may conclude that
\[\snorm{\omega v_j}_2\ge\frac{|\rho|}{16\delta K'+16}\]
for all $\omega\in\mb{C}$ satisfying $|\omega| = 1$.

Next, by \cref{lem:low-density} in the case $\ell = 2$, applied to $V = [x_j, y_j]^{\mr{T}}$, we have
\[\mb{P}[\mc{E}_{K'}\cap|\sang{M_j,v_j}|\le\epsilon|M_{-j}]\le\frac{en\delta^{-2}K^2}{\det(VV^\mr{T})^{1/2}}\cdot\pi\epsilon^2.\]
This follows since $\langle M_j, v_j \rangle$ (viewed as an $\mb{R}^{2}$-valued random variable) is distributed as a deterministic translation of the $\mb{R}^{2}$-valued random variable $V (\delta (G_n)_j)$, and $\delta (G_n)_j$ is a random vector, each of whose components is a real random variable with density bounded by $Kn\delta^{-1}$.

It remains to estimate $d:=\det(V V^{\mr{T}})^{1/2}$. By \cite[Proposition~B.0.1]{Ge17}, 
we have
\[\min_{\theta\in\mb{R}}\snorm{e^{i\theta}z_j}_2^2 = \frac{1}{2}-\frac{\sqrt{1-4d^2}}{2}\le 2d^2,\]
hence
\[d\ge\frac{|\rho|}{16\sqrt{2}(\delta K'+1)},\]
which completes the proof. 
\end{proof}
\begin{proof}[Proof of \cref{prop:complex-sv}]
\cref{lem:reduction,lem:two-dimensionality} give
\[\mb{P}[\mc{E}_{K'}\cap\sigma_n(M)\le\epsilon]\le n\left(c_{\ref{lem:two-dimensionality}}(1+\delta K')\frac{\max(K,K^2)n(\epsilon n^{1/2})^2}{\delta^2|\on{Im}z|}\right).\qedhere\]
\end{proof}

\section{Proof of \texorpdfstring{\cref{prop:spacing}}{Proposition 4.3}}\label{app:spacing}
The proof of \cref{prop:spacing} follows a similar general outline as in \cite{Ge17}.\\ 

We first require a statement relating relating the property that two eigenvalues are contained within a disk to bounds on the two smallest singular values of an appropriately shifted matrix. This follows immediately from the log-majorization theorem, which implies that $\sigma_n\sigma_{n-1}\le|\lambda_n\lambda_{n-1}|$ \cite{HJ91}. This also follows quickly from \cite[Lemma~A.1,~A.2]{LO20}.
\begin{lemma}[{From \cite[Theorem~3.3.2]{HJ91}}]\label{lem:orthog}
Fix $N\in \mb{C}^{n\times n}$ and $z\in \mb{C}$ with $|z|\le \norm{N}$. Let $M = N - zI$ and suppose there exist $\lambda_i,\lambda_j\in \mc{D}(z,s)$ for $i\neq j$. Then, there exist $s\le t\le 2\norm{N}$ such that
\[\sigma_n(M)\le s^2/t\emph{ and } \sigma_{n-1}(M)\le 2t.\]
\end{lemma}
We now prove a variant of \cref{prop:complex-sv} for the smallest two singular values.
\begin{lemma}\label{lem:complex-sv-2}
Let $A\in\mb{R}^{n\times n}$ with $\snorm{A}\le 1$ and let $\delta \in (0,1)$. Let $z\in \mb{C}$ with $|z|\le 3\delta K'+3$ and  $M = A + \delta G_n-zI$. Let $\mc{E}_{K'}$ denote the event that $\snorm{G_n}\le K'$. Then, 
\[\mb{P}[\sigma_n(M) < \epsilon_1 \cap \sigma_{n-1}(M)\le \epsilon_2 \cap \mc{E}_{K'}]\le c_{\ref{lem:complex-sv-2}}(1+\delta K')^2\max(K^2,K^4)n^6\frac{\epsilon_1^2\epsilon_2^2}{\delta^4(\on{Im}z)^2}\]
where $c_{\ref{lem:complex-sv-2}}$ is an absolute constant.
\end{lemma}
\begin{proof}
Let $G_j$ the principal minor of $G$ excluding the $j$th row and column and let $\mc{E}_{K'}^{j}$ denote the event that $\norm{G_j}\le K'$. Then by \cref{lem:double-reduction} we have that 
\begin{align*}
\mb{P}[\sigma_n(M) < \epsilon_1 & \cap \sigma_{n-1}(M)\le \epsilon_2 \cap \mc{E}_{K'}]\\
&\le \sum_{i=1}^n\sum_{j\neq i}\mb{P}[\on{dist}(M_i,H_i)\le \epsilon_1 n^{1/2}\cap \on{dist}(M_j,H_{i,j})\le \epsilon_2 n^{1/2}\cap \mc{E}_{K'}]\\
& = \sum_{i=1}^n\sum_{j\neq i}\mb{P}[\on{dist}(M_i,H_i)\le \epsilon_1 n^{1/2}\cap \on{dist}(M_j,H_{i,j})\le \epsilon_2 n^{1/2}\cap \mc{E}_{K'} \cap \mc{E}_{K'}^i]\\
&= \sum_{i=1}^n\sum_{j\neq i}\mb{E}\bigg[\mb{P}[ \on{dist}(M_j,H_{i,j})\le \epsilon_2 n^{1/2}\cap \mc{E}_{K'}^i|M_{-\{i,j\}}]\\
&\quad\quad\cdot\mb{P}[\on{dist}(M_i,H_i)\le \epsilon_1 n^{1/2}\cap \mc{E}_{K'}|M_{-\{i,j\}} \cap \on{dist}(M_j,H_{i,j})\le \epsilon_2 n^{1/2}\cap \mc{E}_{K'}^i]\bigg]\\
&\le \sum_{i=1}^n\sum_{j\neq i}\sup_{M_{-\{i,j\}}}\mb{P}[ \on{dist}(M_j,H_{i,j})\le \epsilon_2 n^{1/2}\cap \mc{E}_{K'}^i|M_{-\{i,j\}}]\\
&\quad\quad\cdot\sup_{M_{-i}}\mb{P}[ \on{dist}(M_i,H_i)\le \epsilon_1 n^{1/2}\cap \mc{E}_{K'}|M_{-i}].
\end{align*}
Note that the second term in the product is controlled by \cref{lem:two-dimensionality}, using that $\on{dist}(M_i,H_i) \geq |\sang{M_i,v_i}|$, where $v_i$ is a unit normal to $H_i$.

The first term can be controlled by noting that
\[\on{dist}(M_j,H_{i,j})\ge|\sang{M_j,v_{i,j}}|,\]
where $v_{i,j}$ is a unit normal to $H_{i,j}$ with $i$th coordinate $0$ and then applying \cref{lem:two-dimensionality} to the $(n-1)\times (n-1)$ minor of $M$ formed by excluding the $i$th row and column (we must replace $\delta$ by $\delta(1-n^{-1})^{1/2}$ in this application). 

To summarize, the last term in the chain of inequalities above can be bounded by
\[n(n-1)\left(c_{\ref{lem:two-dimensionality}}(1+\delta K')\frac{\max(K,K^2)n(\epsilon_1n^{1/2})^2}{\delta^2|\on{Im}z|}\right)\left(c_{\ref{lem:two-dimensionality}}(1+\delta K')\frac{\max(K,K^2)n(\epsilon_2n^{1/2})^2}{\delta^2(1-n^{-1})|\on{Im}z|}\right)\]
\[=c_{\ref{lem:two-dimensionality}}^2(1+\delta K')^2\max(K^2,K^4)n^6\frac{\epsilon_1^2\epsilon_2^2}{\delta^4(\on{Im}z)^2},\]
which completes the proof. 
\end{proof}

Finally, we state a variant of \cref{prop:unif-bound} for the smallest two singular values, whose proof follows in an identical manner to the proof of \cref{lem:complex-sv-2}, with  \cref{prop:unif-bound} replacing the application of \cref{prop:complex-sv}. We omit further details. 
\begin{lemma}\label{lem:unif-bound-2}
For any $A\in\mb{C}^{n\times n}$, we have
\[\mb{P}[\sigma_n(A+\delta G_n) < \epsilon_1 \cap \sigma_{n-1}(A+\delta G_n)\le \epsilon_2]\le c_{\ref{lem:unif-bound-2}}K^2n^4\frac{\epsilon_1\epsilon_2}{\delta^2}.\]
\end{lemma}
Now, we are ready to complete the proof of our eigenvalue spacing result.
\begin{proof}[Proof of \cref{prop:spacing}]
Let $M = A+ \delta G_{n}$. We wish to bound the probability that both the events $\mc{E}_{K'}=\{\snorm{G_n}\le K'\}$ and $\eta(M)\le s$ occur.  Note that when this happens, the eigenvalues $\lambda_1,\ldots,\lambda_n$ of $M$ are contained in $\mc{D}(0,1+\delta K')$, and there is some pair $\lambda_i,\lambda_j$ (with $i\neq j$) within a distance $s$ of each other.

Fix some height $H\ge s$ to be chosen later, and cover the disk $\mc{D} = \mc{D}(0,1+\delta K')$ by regions
\[\mc{R}_j = \mc{D}\cap\{\on{Im}z\in[(2j-1)H,(2j+1)H]\}\]
for all integers $j$ with $|j|\le(1+\delta K')/(2H)$. Then, we cover each $\mc{R}_j$ by $O(H(1+\delta K')s^{-2})$ disks of radius $s$, which we denote by $\mc{D}(z_\alpha,s)$ for $\alpha\in A_j$. Thus $|A_j|\lesssim H(1+\delta K')s^{-2}$. Since $\lambda_i$ is in one of these disks, it follows that there exists a pair $i\neq j$ such that $\lambda_i,\lambda_j\in\mc{D}(z_\alpha,2s)$ for some $j$ and some $\alpha\in A_j$.

From this, and by \cref{lem:orthog}, we have
\begin{align*}
\mb{P}[\mc{E}_{K'}\cap\eta(M)\le s]&\le\sum_j\sum_{\alpha\in A_j}\mb{P}[\exists\lambda_i,\lambda_j\in\mc{D}(z_\alpha,2s)]\\
&\le\sum_j\sum_{\alpha\in A_j}\mb{P}\big[\exists t\in[s,2\snorm{M}]: \sigma_n(M-z_\alpha I)\le s^2/t\cap\sigma_{n-1}(M - z_\alpha I)\le 2t\big]\\
&\le\sum_j\sum_{\alpha\in A_j}\sum_{\ell=0}^{\lfloor\log_2(2(1+\delta K')/s)\rfloor}\mb{P}[\sigma_n(M - z_\alpha I)\le 2^{-\ell}s\cap\sigma_{n-1}(M - z_\alpha I)\le 2^{\ell+2}s],
\end{align*}
where the last line represents a dyadic chop on the possible values of $t$, restricting to $t = 2^\ell s$ over the stated range of $\ell$. Finally, we apply \cref{lem:complex-sv-2} and \cref{lem:unif-bound-2}, with the latter used only in $\mc{R}_0$ (i.e., for $j = 0$).\\

For $j\neq 0$ and $\alpha\in A_j$, we have $|\on{Im}z_\alpha|\ge |j|H$. Hence, by \cref{lem:complex-sv-2},
\[\mb{P}[\sigma_n(M-z_\alpha I)\le 2^{-\ell}s\cap\sigma_{n-1}(M-z_\alpha I)\le 2^{\ell+2}s]\le c_{\ref{lem:complex-sv-2}}(1+\delta K')^2\max(K^2,K^4)n^6\frac{16s^4}{\delta^4j^2H^2}.\]

For $\alpha\in A_0$ we have by \cref{lem:unif-bound-2} that
\[\mb{P}[\sigma_n(M-z_\alpha I)\le 2^{-\ell}s\cap\sigma_{n-1}(M-z_\alpha I)\le c_{\ref{lem:unif-bound-2}}K^2n^4\frac{4s^2}{\delta^2}.\]

Putting this together, we obtain
\begin{align*}
\mb{P}[\mc{E}_{K'}\cap\eta(M)\le s]&\le\sum_j\sum_{\alpha\in A_j}\sum_{\ell=0}^{\lfloor\log_2(2(1+\delta K')/s)\rfloor}\mb{P}[\sigma_n(M)\le 2^{-\ell}s\cap\sigma_{n-1}(M)\le 2^{\ell+2}s]\\
&\lesssim\log_2(2(1+\delta K')/s)\bigg(
H(1+\delta K')s^{-2}\cdot\left(K^2n^4\frac{s^2}{\delta^2}\right)\\
&\qquad+\sum_{j\neq 0}\left[H(1+\delta K')s^{-2}\cdot\left((1+\delta K')^2\max(K^2,K^4)n^6\frac{s^4}{\delta^4j^2H^2}\right)\right]\bigg)\\
&\lesssim\log_2(2(1+\delta K')/s)\left((1+\delta K')^3\max(K^2,K^4)n^6\frac{s^2}{\delta^4H}+(1+\delta K')K^2n^4\frac{H}{\delta^2}\right),
\end{align*}
where $\lesssim$ indicates a suppressed absolute constant. 

Finally, we choose
\[H = (1+\delta K')\max(1,K)\frac{sn}{\delta}\ge s\]
to find that
\[\mb{P}[\mc{E}_{K'}\cap\eta(M)\le s]\lesssim\log_2(2(1+\delta K')/s)\cdot (1+\delta K')^2\max(K^2,K^3)n^5\frac{s}{\delta^3}.\qedhere\]
\end{proof}

\bibliographystyle{amsplain0.bst}
\bibliography{main.bib}

\appendix
\section{Proof of \texorpdfstring{\cref{lem:double-reduction}}{Lemma 6.2}}\label{app:two-singular-values}

We begin with a preliminary lemma.
\begin{lemma}\label{lem:second-smallest-singular-value}
Let $M$ be any $n\times n$ random matrix and fix any $\epsilon\ge 0$. Fix an index $\ell \in [n]$. Let $M_i, H_i$ be as in \cref{lem:reduction} and let $H_{i,j} = \on{span}\{M_k: k\neq i,j\}$. Then for any event $\mc{E}$,
\[\mb{P}[\sigma_{n-1}(M)\le\epsilon\cap \mc{E}]\le \sum_{j\in[n]\setminus\ell}\mb{P}[\on{dist}(M_j,H_{j,\ell})\le \epsilon n^{1/2}\cap \mc{E}].\]
\end{lemma}
\begin{proof}
If $\sigma_{n-1}(M)\le\epsilon$, then there is a subspace $W$ of dimension $2$ satisfying $\snorm{M|_W}\le\epsilon$. Since $W$ has dimension $2$, there must exist some vector $w\in W$ with $w_\ell = 0$ and $\snorm{w}_2=1$. In particular, there exists some $j\neq \ell$ such that $|w_j|\ge n^{-1/2}$. 
Then
\[\epsilon\ge\snorm{Mw}_2\ge\on{dist}(Mw,H_{j,\ell}) = \on{dist}(M_jw_j+M_\ell w_\ell,H_{j,\ell}) = \on{dist}(M_jw_j,H_{j,\ell})\ge n^{-1/2}\on{dist}(M_j,H_{j,\ell}).\]

We have shown that if the event $\sigma_{n-1}(M)\le\epsilon \cap \mc{E}$ occurs, then at least one of the events $\on{dist}(M_j,H_{j,\ell})\le\epsilon n^{1/2} \cap \mc{E}$ occurs for $j\neq\ell$. This clearly implies the desired result.
\end{proof}
Now we are ready to deduce \cref{lem:double-reduction}.
\begin{proof}[Proof of \cref{lem:double-reduction}]
We have
\begin{align*}
\mb{P}[\sigma_n(M)\le\epsilon_1\cap&\sigma_{n-1}(M)\le\epsilon_2\cap \mc{E}]\\
&\le\sum_{i=1}^n\mb{P}[\on{dist}(M_i,H_i)\le\epsilon_1n^{1/2}\cap\sigma_{n-1}(M)\le\epsilon_2\cap \mc{E}]\\
&\le\sum_{i=1}^n\sum_{j\neq i}\mb{P}[\on{dist}(M_i,H_i)\le\epsilon_1n^{1/2}\cap\on{dist}(M_j,H_{i,j})\le\epsilon_2n^{1/2}\cap\mc{E}];
\end{align*}
the first inequality follows by applying \cref{lem:reduction} with $\epsilon = \epsilon_1$ and $\mc{E}$ replaced by $\sigma_{n-1}(M) \leq \epsilon_2\cap\mc{E}$, and the second inequality follows by 
applying \cref{lem:second-smallest-singular-value} with $\epsilon = \epsilon_{2}$, $\ell = i$, and $\mc{E}$ replaced by $\on{dist}(M_i,H_i)\le\epsilon_1n^{1/2}\cap\mc{E}$.
\end{proof}

\section{Eigenvalue Spacing for Complex Random Variables}\label{app:complex-spacing}
In this appendix, we sketch a series of estimates regarding the eigenvalue spacing for complex perturbations. We first state the analog of \cref{lem:complex-sv-2} in the complex case. Throughout, $G_{n} = G_{n}(\xi)$, where $\xi$ is any sub-Gaussian complex random variable with bounded density, as in \cref{sec:general-complex}.
\begin{lemma}\label{lem:complex-sv-3}
Let $A\in\mb{C}^{n\times n}$ and let $M = A + \delta G_n$. Then, \[\mb{P}[\sigma_n(M) \le \epsilon_1 \cap \sigma_{n-1}(M)\le \epsilon_2]\le c_{\ref{lem:complex-sv-3}}K^2n^6\frac{\epsilon_1^2\epsilon_2^2}{\delta^4}\]
where $c_{\ref{lem:complex-sv-3}}$ is an absolute constant.
\end{lemma}
\begin{proof}[Sketch]
This essentially identical to the proof of \cref{lem:complex-sv-2} with \cref{lem:two-dimensionality} replaced by the estimate obtained in \cref{lem:general-small-ball}.
\end{proof}
\begin{theorem}\label{thm:spacing-complex-general}
Let $A\in\mb{C}^{n\times n}$ with $\snorm{A}\le 1$, let $\delta \in (0,1)$, and let $M = A + \delta G_n$. Let $\mc{E}_{K'}$ denote the event that $\snorm{G_n}\le K'$. Then, for $s\le 1$, 
\[\mb{P}[\mc{E}_{K'}\cap\eta(M)\le s]\le c_{\ref{thm:spacing-complex-general}}\log_2(2(1+\delta K')/s)\cdot (1+\delta K')^2K^4n^6\frac{s^2}{\delta^4}\]
where $c_{\ref{thm:spacing-complex-general}} > 0$ is an absolute constant.
\end{theorem}
\begin{proof}[Sketch]
This follows by taking an $s$-net of $\mc{D}(0,1+\delta K')$, and then following the proof of \cref{prop:spacing}, using \cref{lem:complex-sv-3} instead of \cref{lem:complex-sv-2}.
\end{proof}
Using the above theorem, one can obtain a high-probability spacing between the eigenvalues of $A + \delta G_{n}(\xi)$ on the order of $\delta^2n^{-3}/\log(n\delta^{-1})$; for $\delta$ sufficiently small, this is sharper than the spacing of order $(\delta/n)^{8/3}$ obtained for $A + \delta G_{n}(\mc{N}_{\mb{C}}(0,1))$ by Banks, Vargas,  Kulkarni, and Srivastava \cite{BVKS19}. It is quite likely that in the complex Gaussian case, one can obtain a version of \cref{lem:complex-sv-3} which has a better dependence on $n$, and thereby obtain an eigenvalue spacing estimate of the order of $(\delta/n)^{2}/\log(n\delta^{-1})$, which represents a strict improvement over the estimate in \cite{BVKS19}. 

\section{Details for \texorpdfstring{\cref{thm:complex}}{Proposition 3.1}}\label{app:details}
Here, we show how to use \cref{lem:operator-tail,lem:complex-final} to complete the proof of \cref{thm:complex}.
\begin{proof}[Proof of \cref{thm:complex}]
Let $M = A + \delta G_n$, with eigenvalues $\lambda_1,\ldots,\lambda_n$. By \cref{lem:operator-tail}, 
\[\mb{P}[\sigma_{n}(G_n)\le K']\ge\frac{3}{4}\]
where $K'$ depends only on the sub-Gaussian moment of $\xi$. Let $\mc{E}_{K'}$ denote the event that $\sigma_{n}(G_n)\le K'$. Then, $\mb{P}[\mc{E}_{K'}]\geq 3/4.$

Furthermore, using \cref{lem:complex-final} for $\mc{B} = \mc{D}(0,1+\delta K')$, we have that 
\[\mb{E}[\sum_{\lambda_i\in \mc{B}}\kappa(\lambda_i)^2]\le \frac{\pi eKn^3(1+\delta K')^2}{\delta^2}.\]
Recall that $K$ here is an upper bound on the density $\xi$ which we have assumed. 

Under the event $\mc{E}_{K'}$ we also have that 
\[\sum_{\lambda_i\in \mc{B}}\kappa(\lambda_i)^2 = \sum_{i=1}^{n}\kappa(\lambda_i)^2 = \kappa_2^2\]
and therefore, through simple conditioning, we obtain that 
\[\mb{E}[\kappa_2^2|\mc{E}_{K'}]\le \frac{4\pi eKn^3(1+\delta K')^2}{3\delta^2}.\]
Hence, using Markov's inequality, we have that 
\[\mb{P}\left[\kappa_2^2\le\frac{4\pi eKn^3(1+\delta K')^2}{\delta^2}\bigg|\mc{E}_{K'}\right]\ge\frac{2}{3},\]
which implies that
\[\mb{P}\left[\kappa_2^2\le\frac{4\pi eKn^3(1+\delta K')^2}{\delta^2}\cap \mc{E}_{K'}\right]\ge\frac{1}{2}.\]
The result now follows immediately from \cref{lem:kappaV-upper}.
\end{proof}

\end{document}